\documentclass[11pt]{amsart}
\usepackage{amsmath,amssymb,amsthm}
\usepackage{graphicx}
\usepackage{tikz}
\usepackage{tikz-cd}
\usepackage[hidelinks]{hyperref}

\newtheorem{theorem}{Theorem}[section]
\newtheorem{lemma}[theorem]{Lemma}

\newtheorem{thmA}{Theorem}

\theoremstyle{definition}
\newtheorem{definition}[theorem]{Definition}
\newtheorem{example}[theorem]{Example}
\newtheorem{remark}[theorem]{Remark}

\newcommand{\Chat}{\widehat{\mathbb{C}}}
\title[Non-uniqueness of characteristic curves]{Non-uniqueness of
characteristic curves of folding rational maps}

\author[L.~Yang]{Luxian Yang}
\address{School of Mathematical Sciences, Shenzhen University,
Shenzhen 518061, People's Republic of China}
\email{lxyang@szu.edu.cn}

\author[J.~Zeng]{Jinsong Zeng}
\address{School of Mathematical Sciences, Shenzhen University,
Shenzhen 518061, People's Republic of China}
\email{jinsongzeng@163.com}

\begin{document}

\begin{abstract}
A folding rational map is a postcritically finite rational map $f$ admitting an essential Jordan curve $\beta$, disjoint from the postcritical set $P_f$, whose preimage is disconnected and consists of Jordan curves homotopic to $\beta$ relative to $P_f$. Such $\beta$ is called a characteristic curve of $f$. We construct a real postcritically finite rational map $f$ of degree $6$ with $\#P_f=4$, which admits two non-homotopic characteristic curves $\beta$ and $\beta'$. It is hyperbolic with orbifold signature $(2,2,3,\infty)$. This non-uniqueness is not a Latt\`es phenomenon, nor does it arise from a mating decomposition. The map is constructed piecewise, and a numerical model is provided to illustrate its dynamics. Thurston obstructions are excluded by arc-lifting and intersection-number arguments.
\end{abstract}

\maketitle

\section{Introduction}\label{sec:intro}

\subsection{Background}
\label{sec:background}

Let $f:\Chat\to\Chat$ be a rational map of degree at least two. Denote by $\operatorname{Crit}(f)$ its set of critical points and by
\[
  P_f:=\overline{\bigcup_{n\ge1} f^{\circ n}(\operatorname{Crit}(f))}
\]
its postcritical set. The map $f$ is called \emph{postcritically finite} if $P_f$ is a finite set. A general principle in the study of rational dynamics is to decompose a dynamical system into invariant subsystems with simpler behavior. For postcritically finite maps, a natural way to do this is to cut along Jordan curves which are invariant up to homotopy rel $P_f$; the first return map on a periodic piece gives rise to a decomposition of $f$.

The classical instance of this principle is the notion of an equator, namely an essential Jordan curve $\gamma\subset\Chat\setminus P_f$ whose preimage $f^{-1}(\gamma)$ is again a single Jordan curve homotopic to $\gamma$ rel $P_f$. Cutting along the equator decomposes the rational map (or, if necessary, its second iterate) into two polynomials, and the map is the mating of the two polynomials \cite{Tan1992,Shi2000}.

Beyond the equator case, invariant curve systems for general postcritically finite rational maps were investigated by Cui, Peng and Tan \cite{CPT2016}. They proved that a non-simply-connected wandering continuum in the Julia set of a postcritically finite rational map is necessarily a Jordan curve, and that the homotopy classes of its forward iterates relative to the postcritical set form a \emph{Cantor multicurve} \cite[Theorem~1.2]{CPT2016}. Conversely, every Cantor multicurve can be promoted to an \emph{exact annular system}, whose non-escapting set contains uncountably many Jordan curves, all but countably many of which are wandering \cite[Theorem~1.1]{CPT2016}; the complementary pieces are points, closures of periodic Fatou domains, or filled Julia sets of renormalizations \cite[Theorem~1.3]{CPT2016}. To produce such maps, they introduced a surgery procedure called \emph{folding}, which turns a postcritically finite polynomial into a postcritically finite rational map admitting a Cantor multicurve consisting of a single Jordan curve \cite[Theorem~1.5]{CPT2016}. The folding construction is formalized by the following notion, which is the central object of the present paper. 

Recall that a Jordan curve $\gamma$ in $\Chat\setminus P_f$ is called \emph{essential} if each of the two components of $\Chat\setminus\gamma$ contains at least two points of $P_f$. By a \emph{Thurston map} we mean an orientation-preserving postcritically finite topological branched covering over $S^2$.

\begin{definition}[Folding map {\rm \cite[\S8.1]{CPT2016}}]\label{def:folding}
Let $F:\Chat\to\Chat$ be a Thurston map and let $\beta$ be an essential Jordan curve in $\Chat\setminus P_F$. The pair $(F,\beta)$ is called a \emph{folding map} if every component of $F^{-1}(\beta)$ is homotopic to $\beta$ rel $P_F$ and $F^{-1}(\beta)$ has at least two components. A postcritically finite rational map with such a curve $\beta$ is called a \emph{folding rational map}. 
\end{definition}
If a folding map $F$ is Thurston equivalent to a rational map $f$ via $(h_0, h_1)$, then clearly $(f,h_0(\beta))$ is a folding rational map. 

For a folding map $(F,\beta)$, the curve $\beta$ is called a \emph{characteristic curve}. We denote by $m(F,\beta)$ the number of components of $F^{-1}(\beta)$. By definition, $m(F,\beta)\ge2$.

Before discussing the uniqueness of characteristic curves of folding rational maps, we illustrate the definition with a concrete family in which the folding structure can be seen directly.

\begin{example}[Necklace maps]\label{ex:necklace}
In the family of \emph{McMullen maps} \cite{McM1988}
\[
  f_\lambda(z)=z^{n}+\frac{\lambda}{z^{n}},
  \qquad n\ge3,\ \lambda\in\mathbb{C}^{*},
\]
we may choose suitable $\lambda$ such that its finite critical values $\pm v:=\pm2\sqrt{\lambda}$ are preperiodic and lie on the boundary of the Fatou domain containing $0$. Then the Julia set of $f_{\lambda}$ is a necklace \cite[Theorem~1.8]{GYZ2026}, see Figure~\ref{fig:necklace}.  Moreover, the Fatou domains $B$ containing $\infty$ and $T$ containing $0$ are Jordan domains with disjoint closures. The annulus $A:=\Chat\setminus\overline{B\cup T}$ contains no point of $P_{f_\lambda}$, and $f_\lambda^{-1}(A)$ consists of exactly two sub-annuli $A_1,A_2$, each covering $A$ with degree $n$. A core curve $\beta$ of $A$ separates $\{\pm v\}$ from $\{\infty\}\cup \{f_\lambda^{\circ k}(\pm v):k\ge1\}$, and is therefore essential. The two components $\beta_1,\beta_2$ of  $f_\lambda^{-1}(\beta)$ are core curves of $A_1,A_2$ respectively. Then both $\beta_1$ and $\beta_2$ are homotopic to $\beta$ rel $P_{f_\lambda}$. Hence, $(f_\lambda,\beta)$ is a folding map with $m(f_\lambda,\beta)=2$. 
\end{example}

\begin{figure}[htbp]
  \centering
  \includegraphics[width=0.6\textwidth]{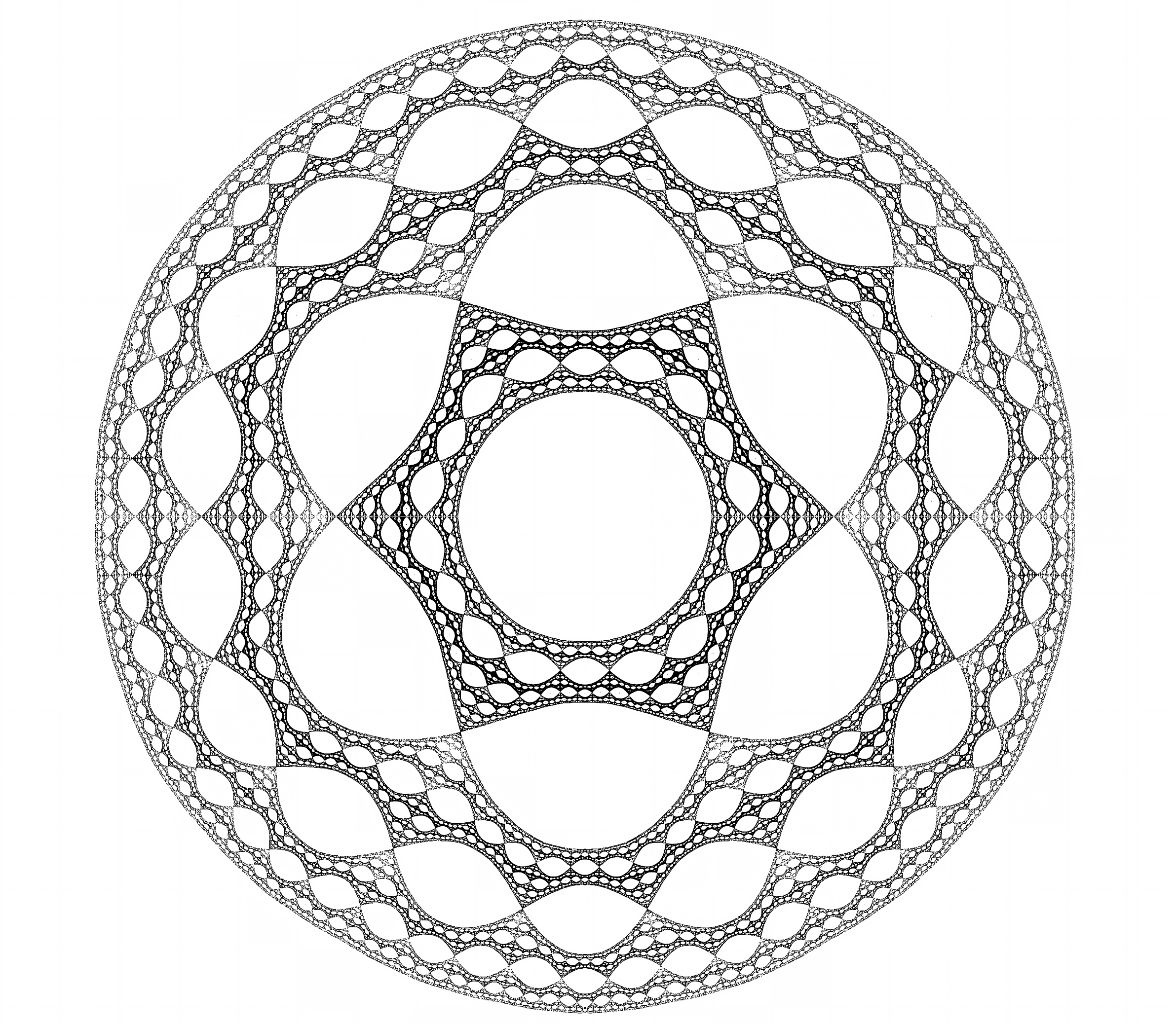}
  \caption{The necklace Julia set of the postcritically finite McMullen map $f_\lambda(z)=z^3+\lambda/z^3$ with $\lambda\approx0.01645$. }
  \label{fig:necklace}
\end{figure}

The present paper concerns the \emph{uniqueness} of the characteristic curve of folding rational maps. A rational map may admit several different mating decompositions, each carrying its own equator, and equators coming from different decompositions are not homotopic relative to the postcritical set \cite{ShiTan2000,Milnor2004,Rees2010}.  This shows that for rational maps arising from matings, equators are not unique up to homotopy. For folding rational maps, flexible Latt\`es maps induced by integer multiplication on a torus provide trivial examples with two non-homotopic characteristic curves \cite{DH1993,Milnor2006}. 

We are thus led to the basic question of this paper.

\begin{quote}
  \emph{Apart from Latt\`es maps, is the characteristic curve of a folding rational map unique up to homotopy relative to its postcritical set?}
\end{quote}

\subsection{Main results}\label{sec:main-result}

Our main result answers this question in the negative.

\begin{thmA}\label{thm:A}
There exists a real rational map $f:\Chat\to\Chat$ of degree $6$ satisfying the following properties.
\begin{enumerate}
  \item $f$ is postcritically finite with $P_f=\{0,v_0,v_1,v_2\}$. The point $0$ is a superattracting fixed point of local degree $3$, and $f^{-1}(0)=P_f$.
  \item $f$ admits two characteristic curves $\beta$ and $\beta'$:  $f^{-1}(\beta)$ (resp. $f^{-1}(\beta')$) consists of two essential Jordan curves, homotopic to $\beta$ (resp.\ $\beta'$) rel $P_f$, the curves map with degrees $4$ and $2$.
  \item $\beta$ and $\beta'$ are not homotopic rel $P_f$. More precisely the curve $\beta$ separates $\{0,v_0\}$ from $\{v_1,v_2\}$, whereas $\beta'$ separates $\{0,v_1\}$ from $\{v_0,v_2\}$.
  \item The orbifold of $f$ has signature $(2,2,3,\infty)$. In particular $f$ is not a Latt\`es map.
\end{enumerate}
\end{thmA}

\begin{figure}[htbp]
    \centering
    \begin{minipage}[c]{0.4\textwidth}
        \centering
        \includegraphics[width=\linewidth]{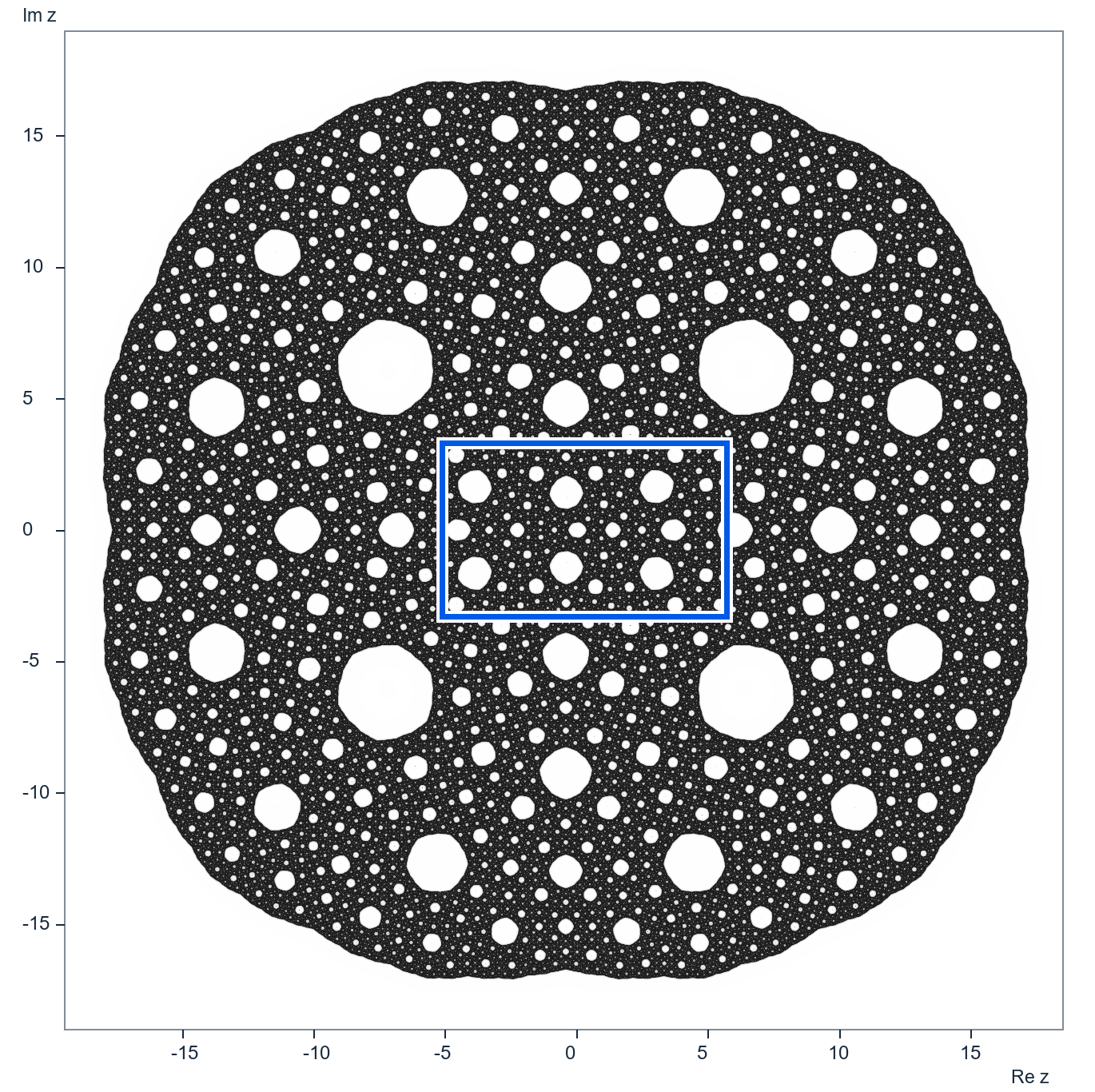}
    \end{minipage}
    \begin{minipage}[c]{0.56\textwidth}
        \centering
        \includegraphics[width=\linewidth]{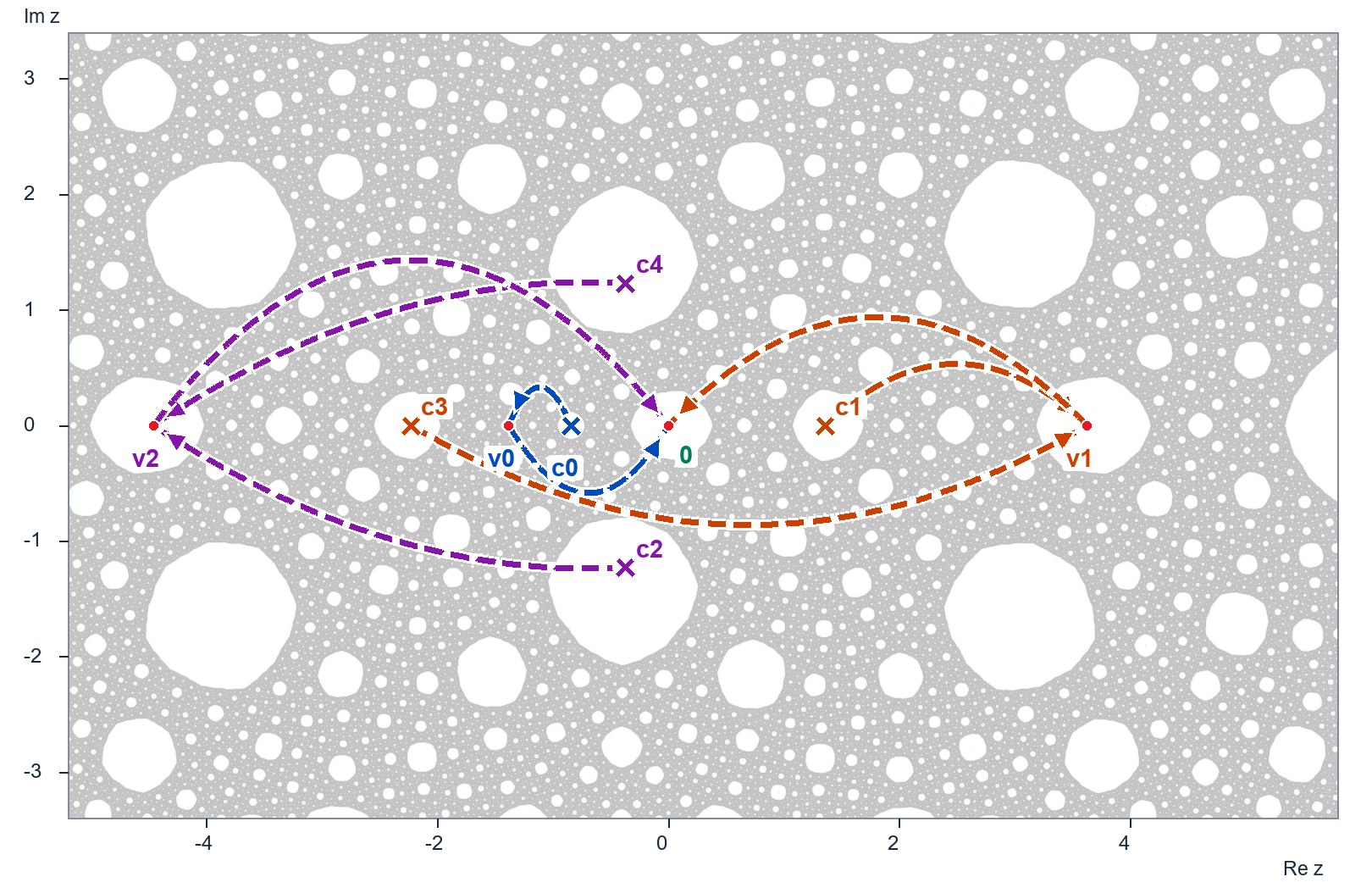}
        \includegraphics[width=\linewidth]{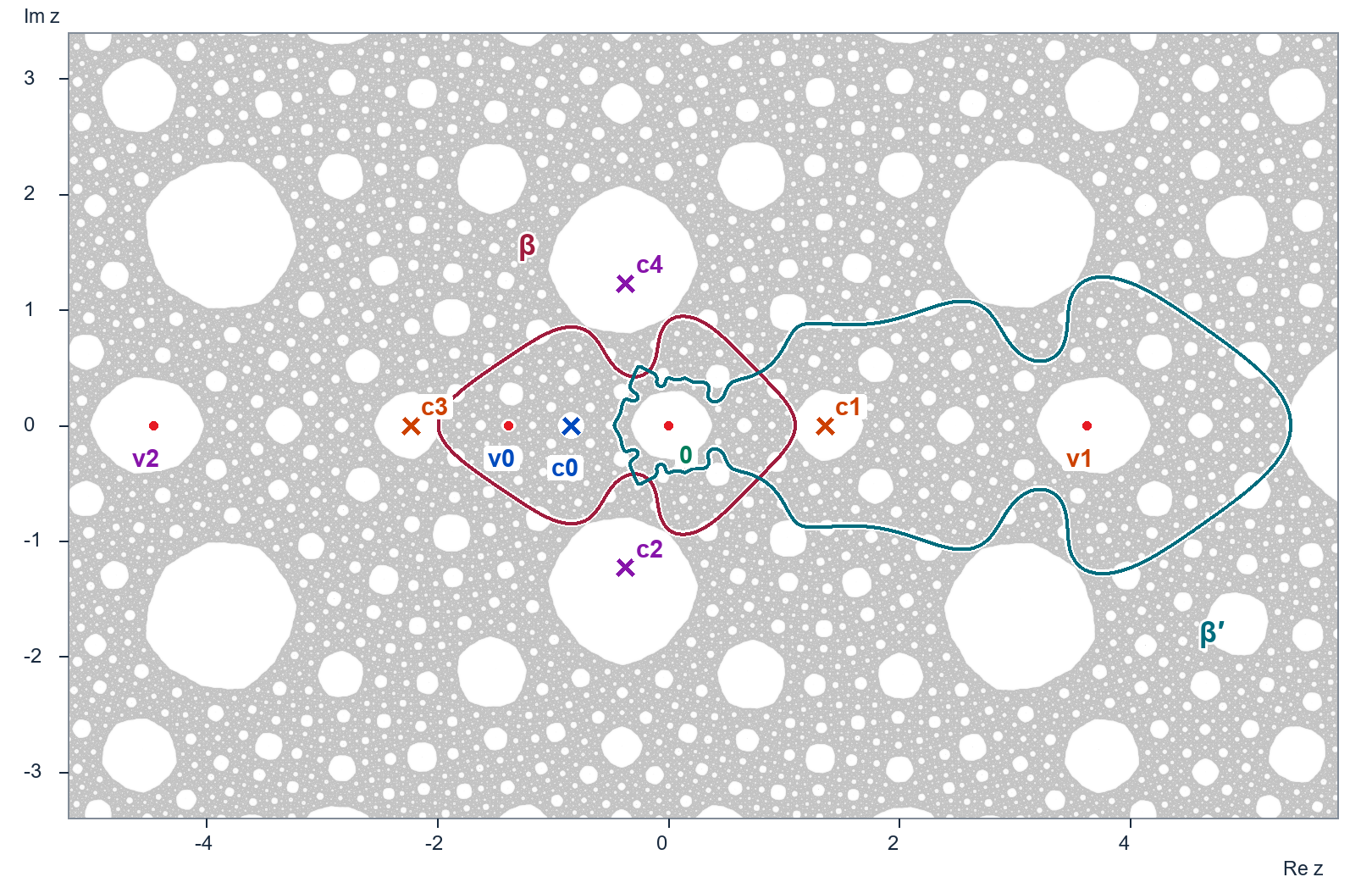}
    \end{minipage}
    \caption{The numerical model of $f$. Left: the Julia set of $f$. Top right: a magnification of the blue rectangle with the critical orbit marked. Bottom right: a magnification of the blue rectangle with the curves $\beta$ and $\beta'$ marked.}
    \label{fig:three-images}
\end{figure}

A numerical model of $f$ is given in \S\ref{sec:formula} (See Figure \ref{fig:three-images}). 

The folding rational maps from \cite[\S8]{CPT2016} and the McMullen examples above each have only one characteristic curve. Theorem~\ref{thm:A} shows that the uniqueness fails in general. If we relax the rationality requirement and consider branched coverings, then a slight modification of the example in \cite[\S 8.4]{CPT2016} yields branched coverings admitting two non-homotopic characteristic curves. All branched coverings obtained in this way have Thurston obstructions. In the present paper we construct a branched covering without Thurston obstructions.

Since the map of Theorem~\ref{thm:A} has exactly four postcritical points, it is natural to ask whether it already follows from known realization results. Nearly Euclidean Thurston maps, i.e., Thurston maps with four postcritical points and local degree two at every critical point, form a class for which invariant curves and Thurston obstructions have been analyzed systematically \cite{CFPP2012}; at the other extreme, expanding rational Thurston maps have Julia set equal to the whole sphere \cite{BM2017}.  Related examples with two invariant curve classes occur in blown-ups of Lattès maps \cite{BHI2024}. In those examples the distinguished horizontal and vertical curves have non-essential preimages. 

The map of Theorem~\ref{thm:A} lies beyond the reach of these results: its orbifold has signature $(2,2,3,\infty)$, so it is not nearly Euclidean; it has a superattracting fixed point, so it is not expanding; and the preimages of characteristic curves consist of essential curves in contrast to \cite{BHI2024}. To the best of our knowledge, no previously known result yields a folding rational map with two non-homotopic characteristic curves. 

\subsection{Strategy of the proof}\label{sec:strategy}

The proof of Theorem~\ref{thm:A} proceeds as follows. We first construct a Thurston map $F$ of degree $6$, defined piecewise on a decomposition of the sphere into two disks and an annulus: a quartic piece $g_1$ transplanted from an explicit polynomial model, a quadratic branched covering $g_2$ from the unbounded disk onto the bounded one, and an annulus piece $g_3$ of degree $6$ containing two sub-annuli mapping with degrees $4$ and $2$, glued along the boundary circles. 

The curves $\beta$ and $\beta'$ are realized by the two separation patterns $\{0,v_0\}\mid\{v_1,v_2\}$ and $\{0,v_1\}\mid\{v_0,v_2\}$ of $P_F=\{0,v_0,v_1,v_2\}$, and the construction provides, for each of them, two preimage components of degrees $4$ and $2$, all homotopic to the curve relative to $P_F$ (Lemma~\ref{lem:beta-prime}). Since $\#P_F=4$, every multicurve consists of a single curve, so a Thurston obstruction would be given by one essential curve. The three possible separation patterns of $P_F$ are excluded by arc-lifting and intersection-number arguments, the key input being the fibre $F^{-1}(v_2)=\{c_2,c_4\}$ with local degree $3$ at both points. 

Then Thurston's theorem \cite{DH1993} provides a rational map $f$, unique up to M\"obius conjugacy. The symmetry of $F$ under complex conjugation yields real coefficients. The construction and the proof are carried out in detail in \S\ref{sec:pieces}--\ref{sec:realization}.

\section{Construction and realization of $F$}\label{sec:construction}

In this section we first fix the terminology from Thurston's theory used below, then construct the Thurston map $F$ of Theorem~\ref{thm:A} piece by piece,  and finally show that $F$ admits no Thurston obstructions.

\subsection{Thurston theory}\label{sec:thurston}
Two Thurston maps $F$ and $G$ are \emph{Thurston equivalent} if there are orientation-preserving homeomorphisms $h_0,h_1:(S^{2},P_F)\to(S^{2},P_G)$, isotopic rel $P_F$, such that $h_0\circ F=G\circ h_1$.

A \emph{multicurve} is a collection $\Gamma=\{\gamma_1,\dots,\gamma_n\}$ of pairwise disjoint, pairwise non-homotopic essential Jordan curves in $S^{2}\setminus P_F$. It is called \emph{$F$-invariant} if every essential component of $F^{-1}(\Gamma)$ is homotopic rel $P_F$ to a curve in $\Gamma$. The \emph{Thurston linear transformation} $F_\Gamma:\mathbb{R}^{\Gamma}\to\mathbb{R}^{\Gamma}$ is defined by
\[
  F_\Gamma(\gamma)=\sum_{\delta\in\Gamma}\Biggl(
  \sum_{\substack{\alpha\subset F^{-1}(\gamma)}}
  \frac{1}{\deg(F:\alpha\to\gamma)}\Biggr)\,\delta,
\]
where the inner sum runs over the essential components $\alpha$ of $F^{-1}(\gamma)$ homotopic to $\delta$ rel $P_F$. An invariant multicurve $\Gamma$ is a \emph{Thurston obstruction} if the spectral radius (which is the largest eigenvalue by the Perron-Frobenius
theorem) of $F_{\Gamma}$ is $\ge1$.

The \emph{orbifold} of $F$ is the pair $(S^{2},\nu_F)$, where $\nu_F: P_F\to \mathbb{Z}_{\geq 2}\cup \{\infty\}$: for $p\in P_F$, $\nu_F(p)$ is the least common multiple of the local degrees $\deg(F^{\circ n},z)$ over all $n\ge1$ and all $z$ with $F^{\circ n}(z)=p$. It is \emph{hyperbolic} if its Euler characteristic $2-\sum_{p\in P_F}\bigl(1-1/\nu_F(p)\bigr)$ is negative.

\begin{theorem}[\cite{DH1993}]\label{thm:thurston}
A Thurston map $F$ with hyperbolic orbifold is Thurston equivalent to a rational map if and only if $F$ has no Thurston obstructions. In this case the rational map is unique up to M\"obius conjugacy.
\end{theorem}

In this paper $\#P_F=4$, and obstructions take the following simple form. On the sphere with four punctures, any two disjoint essential Jordan curves are homotopic. Consequently every multicurve is a singleton $\{\gamma\}$, and the spectral radius is
\begin{equation}\label{eq:lambda}
  \lambda(\gamma)=\sum_{\substack{\alpha\subset F^{-1}(\gamma)\\
  \alpha\sim\gamma}}\frac{1}{\deg(F:\alpha\to\gamma)}.
\end{equation}

For two essential Jordan curves $\alpha,\beta$ in $S^{2}\setminus P_F$, their \emph{geometric intersection number} $\textup{i}(\alpha,\beta)$ is the minimum of $\#(\alpha'\cap\beta')$ over all curves $\alpha',\beta'$ homotopic to $\alpha,\beta$ rel $P_F$, respectively. Representatives attaining the minimum are said to be in \emph{minimal position}. By definition $\#(\alpha\cap\beta)\ge \textup{i}(\alpha,\beta)$, with equality in minimal position. On the sphere with four punctures, $\textup{i}(\alpha,\beta)=0$ only when $\alpha$ and $\beta$ are homotopic rel $P_F$, since disjoint essential curves are homotopic to each other.

\subsection{The three pieces}\label{sec:pieces}

We now work on the Riemann sphere $\Chat$ and write $\sigma(z)=\bar z$ for complex conjugation. Its fixed point set is the \emph{real circle} $\widehat{\mathbb{R}}=\mathbb{R}\cup\{\infty\}$. All disks, annuli and maps constructed below are chosen to be $\sigma$-symmetric.

The sphere is decomposed as
\[
  \Chat=D_2\cup A\cup D_1,
\]
where $D_1$ is a bounded closed Jordan disk containing $0,v_0$ in its interior (defined through the quartic model in \eqref{eq:Delta0} below) and $D_2$ is an unbounded closed Jordan disk containing $\infty,v_1,v_2$, with $D_1\cap D_2=\emptyset$. The annulus $A$ is bounded by $\partial D_1$ and $\partial D_2$. The map $F$ will be defined piecewise on the three parts: a quartic piece $g_1: D_1\to D$, where $D$ is a closed Jordan disc containing $D_1$ in its interior, a quadratic piece $g_2:D_2\to D$ and an annulus piece $g_3:A\to\Chat\setminus D$.

The Jordan curve
\[
  \beta:=\partial D
\]
separates $\{0,v_0\}$ from $\{\infty,v_1,v_2\}$. It meets the real circle in two points $b^-<0<b^+$, with $b^-\in (v_2,v_0)$ and $b^+\in (0,v_1)$.

\subsection*{The quartic piece $g_1$.}
The piece $g_1$ is transplanted from an explicit polynomial model. Consider the quartic polynomial with real coefficients
\begin{equation}\label{eq:model}
  p(z)=z^{3}\Bigl(z+\frac{4\sqrt[3]{4}}{3}\Bigr). 
 % p'(z)=4z^{2}\bigl(z+\sqrt[3]{4}\bigr).
\end{equation}
Its critical points are $0$, which is a superattracting fixed point of local degree $3$, and $-\sqrt[3]{4}$, which is a preperiodic point of local degree $2$. Hence, $p$ is postcritically finite with connected Julia set. Its filled Julia set is shown in Figure~\ref{fig:g0}. 
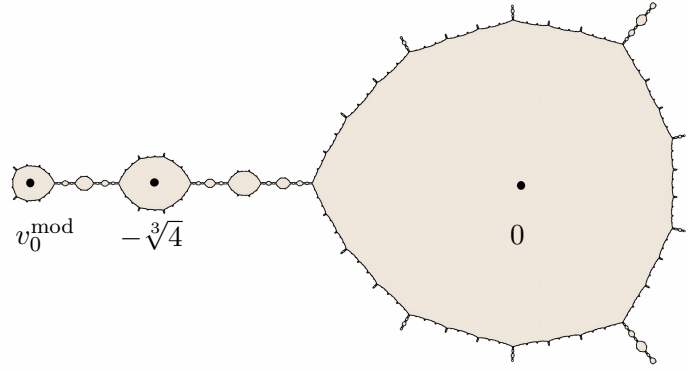
\begin{figure}[htbp]
\centering
\begin{tikzpicture}
\node at (0,0){\includegraphics[width=0.8\textwidth]{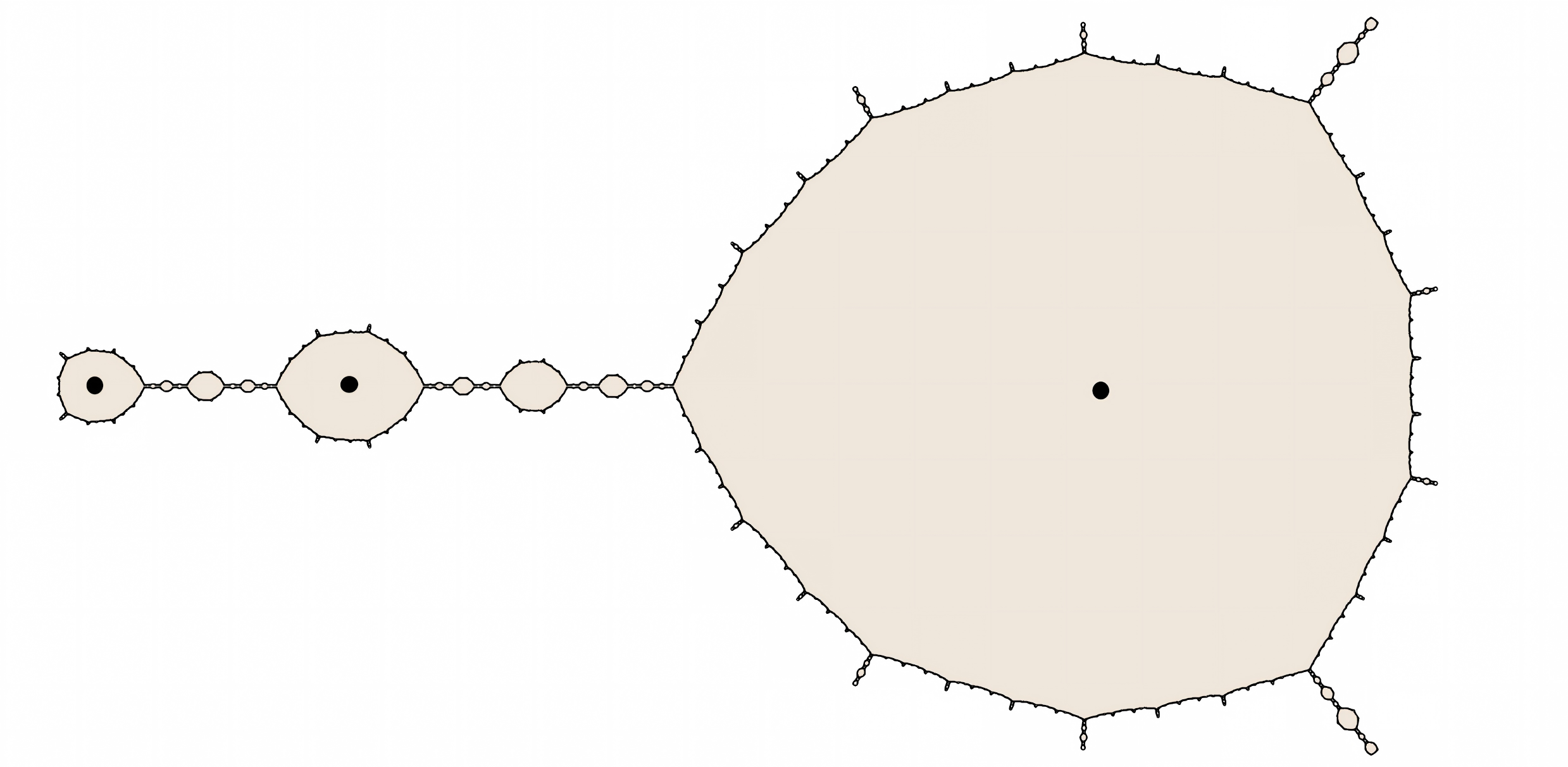}};
\node at (2., -0.7){$0$};
\node at (-4.25, -0.7){$v^{\textup{mod}}_0$};
\node at (-2.85, -0.7){$-\sqrt[3]{4}$};
%\ruler{5}{2}
\end{tikzpicture}
\caption{The filled Julia set of the quartic model $p$ of \eqref{eq:model}. The point $0$ is a superattracting fixed point of local degree $3$. The real critical point $-\sqrt[3]{4}$ satisfies $p(-\sqrt[3]{4})=v_0^{\mathrm{mod}}$ and $p(v_0^{\mathrm{mod}})=0$.}
\label{fig:g0}
\end{figure}

Let $G_p$ denote the Green function of the filled Julia set $K(p)$, extended by $0$ on $K(p)$. For a sufficiently small $\rho>0$, set
\begin{equation}\label{eq:Delta0}
  \Delta:=\{z\in\mathbb{C}:G_p(z)\leq \rho\},\quad
  \Delta_1:=p^{-1}(\Delta)=\{z\in\mathbb{C}:G_p(z)\leq\rho/4\}.
\end{equation}
Since the Julia set is connected, $\Delta$ and $\Delta_1$ are closed Jordan disks bounded by equipotentials, with $K(p)\subset\Delta_1\Subset\Delta$. The restriction $p:\Delta_1\to\Delta$ is a proper branched covering of degree $4$, mapping $\partial\Delta_1$ onto $\partial\Delta$ with degree $4$, and $0,-\sqrt[3]{4},v_0^{\mathrm{mod}}\in\Delta_1$.

The piece $g_1$ is the topological transplant of this model: choose a $\sigma$-equivariant homeomorphism $H_0:D\to\Delta$ identifying the marked points $0,v_0\in D$ with the model points $0,v_0^{\mathrm{mod}}\in\Delta$ such that $D_1=H_0^{-1}(\Delta_1)$, and define $g_1:=H_0^{-1}\circ p\circ H_0\,:\ D_1\to D$. Then $g_1$ is a proper branched covering of degree $4$. It has two critical points. One is $0$, which is a superattracting fixed point of local degree $3$. The other one is $c_0:=H_0^{-1}(-\sqrt[3]{4})$, which is a preperiodic point of local degree $2$. Moreover $g_1(c_0)=v_0$ and $g_1(v_0)=0$. In angle coordinates $\theta\in\mathbb{R}/2\pi\mathbb{Z}$ on the boundary circles, the boundary restriction $g_1:\partial D_1\to\beta=\partial D$ is the degree $4$ covering $\theta\mapsto4\theta$.

\subsection*{The quadratic piece $g_2$.}
The disk $D_2$ is unbounded. In symmetric coordinates we take $D_2=\{z\in\mathbb{C}:|z|\ge R\}\cup\{\infty\}$ with $R>0$. Choose the marked points $v_1,v_2$ in $\rm{int}(D_2)$ so that  $v_2<0<v_1$. Let $g_2:D_2\longrightarrow D$ be a $\sigma$-equivariant proper branched covering of degree $2$ whose unique critical point is $\infty$, of local degree $2$, with $g_2(\infty)=v_0$, and such that $g_2^{-1}(0)=\{v_1,v_2\}$. Such a covering exists: it is modeled on $z\mapsto z^{2}$ between disks and its boundary restriction $g_2:\partial D_2\to\beta$ is the degree $2$ circle covering $\theta\mapsto2\theta$ in angle coordinates.

\subsection*{The annulus piece $g_3$. }
The annulus piece $g_3:A\to\Chat\setminus D$ is constructed as follows (see Figure~\ref{fig:g2}): Choose a $\sigma$-invariant closed Jordan disk $B\subset \rm{int}(D_2)$ whose boundary passes through the marked points $v_1$ on the right and $v_2$ on the left. Then $W:=\Chat\setminus (B\cup D)$ is an annulus bounded by $\beta:=\partial D$ and by $\partial B$.

\begin{figure}[htbp]
\centering
\begin{tikzpicture}
\node at (0,0){\includegraphics[width=0.7\textwidth]{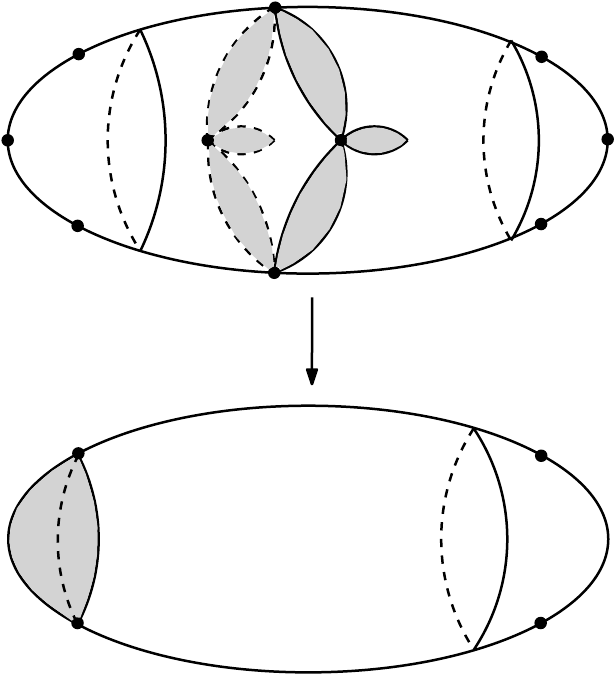}};
\node at (4.65, 2.85){$c_0$};
\node at (-4.6, 2.85){$\infty$};
\node at (-3.5, 2.85){$D_2$};
\node at (-3.5, 4.35){$v_1$};
\node at (-3.5, 1.3){$v_2$};
\node at (-0.5, 5.05){$c_1$};
\node at (-0.5, 0.6){$c_3$};
\node at (-1.75, 2.85){$c_4$};
\node at (0.15, 2.85){$c_2$};
\node at (3.8, 2.85){$D_1$};
\node at (3.5, 4.35){$0$};
\node at (3.5, 1.3){$v_0$};
\node at (-3.5, -4.4){$v_2$};
\node at (-3.5, -1.35){$v_1$};
\node at (3.5, -4.4){$v_0$};
\node at (3.5, -1.35){$0$};
\node at (3.5, -2.85){$D$};
\node at (0.4,0){$g_3$};
\node at (-3.9, -2.85){$B$};
\node at (0.2, 3.75){$L_{12}$};
\node at (0.2, 1.85){$L_{23}$};
\node at (-1, 1.85){$L_{34}$};
\node at (-1, 3.75){$L_{41}$};
\node at (0.9, 2.8){$E^+$};
\node at (-0.9, 2.8){$E^-$};
\node at (2.4, -1){$b^+$};
\node at (2.4, -4.8){$b^-$};
\node at (3, 4.6){$x^+$};
\node at (3, 1.2){$x^-$};
\node at (-2.5, 4.7){$y^+$};
\node at (-2.5,1){$y^-$};
\node at (2.5,-3){$\beta$};
%\ruler{5}{5}
\end{tikzpicture}
\caption{
The annulus piece $g_3:A\to\Chat\setminus D$.}
\label{fig:g2}
\end{figure}

Inside $A$, let $S$ be the union of six closed Jordan disks arranged in a ring as in the top of Figure~\ref{fig:g2}: four \emph{lens disks} $L_{12},L_{23},L_{34},L_{41}$ and two \emph{ear disks} $E^+,E^-$ interchanged by $\sigma$. Consecutive lens disks meet at single points
\[
  L_{41}\cap L_{12}=\{c_1\},\quad
  L_{12}\cap L_{23}=\{c_2\},\quad
  L_{23}\cap L_{34}=\{c_3\},\quad
  L_{34}\cap L_{41}=\{c_4\},
\]
with $c_1,c_3\in\widehat{\mathbb{R}}$ and $c_4=\sigma(c_2)$. Moreover $E^+$ passes through $c_2$ and $E^-$ through $c_4$. In particular $c_1,c_2,c_3,c_4\in\partial S$. The complement $A\setminus S$ has two components, both annuli: the outer one $A_2$, adjacent to $\partial D_2$, and the inner one $A_1$, adjacent to $\partial D_1$.

The map $g_3$ is constructed in two steps. All the maps are chosen $\sigma$-equivariant to guarantee the uniqueness. First, send each of the six disks of $S$ homeomorphically onto $B$: one prescribes the vertex images $c_1,c_3\mapsto v_1$ with the local model $z^2$, $c_2,c_4\mapsto v_2$ with the local model $z^3$, extends them to homeomorphisms between the corresponding boundary arcs, and then to homeomorphisms of the closed disks. Note that the restriction of this map to $\partial S$ covers $\partial B$, and it has degree $2$ on $\partial A_2\setminus \partial D_2$ and has degree $4$ on $\partial A_1\setminus \partial D_1$, which agree with the degrees of $g_2|_{\partial D_2}$ and $g_1|_{\partial D_1}$ respectively. Secondly, extend the boundary coverings to unbranched annulus coverings $A_1\to W$ of degree $4$ and $A_2\to W$ of degree $2$. Since both coverings preserve the real line $\widehat{\mathbb{R}}$, each of the two coverings is unique up to isotopy relative to the boundary. They glue to a continuous map $g_3:A\to\Chat\setminus \rm{int}(D)$.

The resulting map $g_3$ is a branched covering of degree $6$: a point of $\rm{int}(B)$ has six preimages, one in each disk of $S$, and a point of $W$ has $6$ preimages, $2$ in $A_2$ and $4$ in $A_1$. Thus,
$$S=g_3^{-1}(B)\quad\textup{and}\quad g_3^{-1}(W)=A_2\cup A_1.$$
Its critical points are $c_1,c_3$ (of local degree $2$), and $c_2,c_4$ (of local degree $3$). The critical values $v_1=g_3(c_1)=g_3(c_3)$ and $v_2=g_3(c_2)=g_3(c_4)$ lie on $\partial B$, and the fibres over them have types $g_3^{-1}(v_1): (2,2,1,1)$ containing $c_1,c_3$ and two simple roots and $g_3^{-1}(v_2): (3,3)$ containing $c_2,c_4$.

\subsection{Gluing and first properties}\label{sec:gluing}

Define $F:\Chat\to\Chat$ by
\[
  F=g_1\ \text{on }D_1,\qquad
  F=g_3\ \text{on }A,\qquad
  F=g_2\ \text{on }D_2.
\]
The three pieces agree on the common boundary circles by construction, so $F$ is continuous. Near a gluing circle the two adjacent pieces are unbranched and map the two half-neighbourhoods of a boundary point onto complementary half-neighbourhoods of its image. Hence, $F$ is a local homeomorphism on $\partial A$, and $F$ is a Thurston map. 

\begin{lemma}\label{lem:F-properties}
The map $F$ constructed above has the following properties.
\begin{enumerate}
  \item $\deg F=6$ and $\operatorname{Crit}(F)=\{0,c_0,c_1,c_2,c_3,c_4,\infty\}$, with local degrees $3,2,2,3,2,3,2$ respectively.
  \item $F^{-1}(0)=\{0,v_0,v_1,v_2\}=P_F$.
  \item $F$ commutes with $\sigma$ and $F(\widehat{\mathbb{R}})=[v_0,v_1]$. On $\widehat{\mathbb{R}}$, the points appear in the cyclic order: $\infty$, $v_2<c_3<v_0<c_0<0<c_1<v_1$.
  \item $F^{-1}(\beta)=\beta_1\cup\beta_2$, where $\beta_1=\partial D_1$ and $\beta_2=\partial D_2$. Both components are homotopic to $\beta$ rel $P_F$, with $\deg(F,\beta_1)=4$ and $\deg(F,\beta_2)=2$. Thus, $(F,\beta)$ is a folding map with $m(F,\beta)=2$.
  \item The orbifold of $F$ has signature $\bigl(\nu_F(v_0),\nu_F(v_1),\nu_F(v_2),\nu_F(0)\bigr)=(2,2,3,\infty)$. In particular, it is hyperbolic.
\end{enumerate}
\end{lemma}

\begin{proof}
(1) Counting preimages of a regular value in the annulus $W$ shows $\deg F=6$. The pieces are unbranched near the gluing circles, so $\operatorname{Crit}(F)$ is the union of the critical sets of the three pieces: $0$ and $c_0$ from $g_1$, $\infty$ from $g_2$, and $c_1,c_2,c_3,c_4$ from $g_3$, with the displayed local degrees.

(2) One computes $F(0)=g_1(0)=0$; $F(c_0)=F(\infty)=v_0$ and $F(v_0)=g_1(v_0)=0$; $F(c_1)=F(c_3)=v_1$ and $F(v_1)=g_2(v_1)=0$; $F(c_2)=F(c_4)=v_2$ and $F(v_2)=g_2(v_2)=0$. Hence every critical orbit lands on the fixed point $0$, and $P_F=\{0,v_0,v_1,v_2\}$.

(3) Each piece was chosen $\sigma$-equivariant and the gluing respects $\sigma$, so $F$ commutes with $\sigma$. The construction is arranged so that, traveling along the real circle from $\infty$ through the left half, one meets in turn $v_2$, $c_3$, $v_0$, $c_0$, $0$, $c_1$ and $v_1$. This is the displayed cyclic order. 

To prove $F(\widehat{\mathbb{R}})=[v_0,v_1]$, we first work in the model coordinates of the quartic piece. Write $\{u^-,u^+\}=\partial\Delta_1\cap\mathbb{R}$ with $u^-<0<u^+$. Since $v_0^{\mathrm{mod}}=-\tfrac{4\sqrt[3]{4}}3$ lies in $\Delta_1$, we have $u^-<-\tfrac{4\sqrt[3]{4}}3$, hence $p(u^-)=(u^-)^{3}\bigl(u^-+\tfrac{4\sqrt[3]{4}}3\bigr)>0$ and $p(u^+)>0$. Both values are real points of $\partial\Delta$, hence both equal the right point of $\partial\Delta\cap\mathbb{R}$. Transported back by $H_0^{-1}$, this says that both real points of $\partial D_1$ are mapped by $g_1$ to $b^+$. On $[u^-,-\sqrt[3]{4}]$, the model $p$ decreases from its boundary value to the minimum $v_0^{\mathrm{mod}}=p(-\sqrt[3]{4})$, and on $[-\sqrt[3]{4},u^+]$ it increases back (the point $0$ has odd local degree and is not an extremum). Hence the image of $\widehat{\mathbb{R}}\cap D_1$ under $g_1$ is $[v_0,b^+]$. On $\widehat{\mathbb{R}}\cap D_2$ the map $g_2$ has its unique extremum at $\infty$, with value $v_0$, and sends the two real boundary points to $b^+$ and its image is again $[v_0,b^+]$. 

Finally, $\widehat{\mathbb{R}}\cap A$ consists of two arcs, containing $c_1$ and $c_3$ respectively. On each of them $g_3$ rises from $b^+$ at the endpoints to the maximum $v_1$ at the critical point and returns, so the image is $[b^+,v_1]$. Taking the union of the three images, $F(\widehat{\mathbb{R}})=[v_0,b^+]\cup[b^+,v_1]=[v_0,v_1]$.

(4) The map $g_3$ maps onto $\Chat\setminus D$ with $g_3^{-1}(\beta)=\partial D_2\cup\partial D_1$ (the outer boundary of $A_2$ and the inner boundary of $A_1$), while $g_1^{-1}(\beta)=\partial D_1$ and $g_2^{-1}(\beta)=\partial D_2$. Hence
\[
  F^{-1}(\beta)=\partial D_2\cup\partial D_1=\beta_1\cup\beta_2,
\]
where $\beta_1:=\partial D_1, \beta_2:=\partial D_2. $ Each of $\beta_1,\beta_2$ separates $\{0,v_0\}$ from $\{v_1,v_2\}$ and is therefore essential. Note that $\beta\subset A$ and $A\cap P_F=\emptyset$. Hence both $\beta_1$ and $\beta_2$ are homotopic to $\beta$ rel $P_F$. The degrees are $4$ and $2$ by construction, so $(F,\beta)$ is a folding map with $m(F,\beta)=2$.

(5) The fibres of $F$ over the four postcritical points are listed in Remark~\ref{rem:fibres} below. The iterated preimages of $v_0$ and of $v_1$ have local degrees $1$ or $2$ only, so $\nu_F(v_0)=\nu_F(v_1)=2$; those of $v_2$ have local degree $3$ only, so $\nu_F(v_2)=3$; and $\deg(F^{\circ n},0)=3^{n}$ is unbounded, so $\nu_F(0)=\infty$. The Euler characteristic is
$$2-\bigl(\frac12+\frac12+\frac23+1\bigr)=-\frac23<0,$$ so the orbifold is hyperbolic.
\end{proof}

\begin{remark}[Fibres over the postcritical set]\label{rem:fibres}
The fibres of $F$ over the four postcritical points, read off directly from the three pieces, are as follows. The local degrees, in the order in which the points are listed, are given on the right:
\[
\begin{aligned}
  F^{-1}(0)&=\{0,v_0,v_1,v_2\}, &&(3,1,1,1),\\
  F^{-1}(v_0)&=\{c_0,\infty,w_1,w_2\}, &&(2,2,1,1),\\
  F^{-1}(v_1)&=\{c_1,c_3,s_1,s_2\}\subset A, &&(2,2,1,1),\\
  F^{-1}(v_2)&=\{c_2,c_4\}\subset A, &&(3,3).
\end{aligned}
\]
Here $w_1,w_2$ are the two preimages of $v_0$ under $g_1$ different from $c_0$, and $s_1,s_2$ are the two simple points of $g_3^{-1}(v_1)$ in $A$.
\end{remark}

\subsection{The second characteristic curve}\label{sec:beta-prime}

We now construct the second characteristic curve $\beta'$ of Theorem~\ref{thm:A}. Up to homotopy rel $P_F$, it is the boundary of a small neighbourhood of the arc $[0,v_1]$.

\begin{lemma}\label{lem:beta-prime}
There is a Jordan curve $\beta'$ in $\Chat\setminus P_F$, separating $\{0,v_1\}$ from $\{v_0,v_2\}$, such that $F^{-1}(\beta')=\beta'_1\cup\beta'_2$ consists of two Jordan curves, both essential and homotopic to $\beta'$ rel $P_F$, with $\deg(F,\beta'_1)=4$ and $\deg(F,\beta'_2)=2$. In particular $(F,\beta')$ is a folding map with $m(F,\beta')=2$, and $\beta'$ is not homotopic to $\beta$ rel $P_F$.
\end{lemma}

\begin{proof}
\emph{Step 1: the preimage of $\ell:=[0,v_1]$.}
Write $\partial D_1\cap\widehat{\mathbb{R}}=\{x^-,x^+\}$ and $\partial D_2\cap\widehat{\mathbb{R}}=\{y^-,y^+\}$, with $x^-<0<x^+$ and $y^-<0<y^+$. Among the critical values $0,v_0,v_1,v_2$ (Remark~\ref{rem:fibres}) only the endpoints $0,v_1$ lie on $\ell$, and $\ell$ meets $\beta=\partial D$ in the single point $b^+$, since $\beta\cap\widehat{\mathbb{R}}=\{b^-,b^+\}$ and $b^-<v_0<0$. We lift the two subarcs $\ell_0:=[0,b^+]$ and $\ell_1:=[b^+,v_1]$ piece by piece, using the fibres listed in Remark~\ref{rem:fibres}.

We first lift $\ell_0=[0,b^+]$. In the following description,
each lifted arc is understood to include its endpoints.
Under $g_1:D_1\to D$, the arc $\ell_0$ has four lifts with
pairwise disjoint interiors. By the monotonicity of $g_1$
on the real line, two of them are
\[
  [0,x^+]\quad\text{and}\quad[x^-,v_0].
\]
The fibre $g_1^{-1}(0)=\{0,v_0\}$ has local degrees $3$ and $1$.
Thus three lifted arcs meet at $0$, whereas only one ends at $v_0$.
The remaining two lifts therefore join $0$ to the two non-real
points $\zeta,\sigma(\zeta)$ of $g_1^{-1}(b^+)$ on $\partial D_1$.
These two arcs are exchanged by $\sigma$.

Under $g_2:D_2\to D$, the two lifts of $\ell_0$ are the real arcs
\[
  [y^+,v_1]\quad\text{and}\quad[v_2,y^-].
\]
This follows from the monotonicity of $g_2$ on the real circle
established in the proof of Lemma~\ref{lem:F-properties}(3).

We next lift $\ell_1=[b^+,v_1]$, whose interior lies in $W$.
On the right real arc of $A$, the map $g_3$ takes the value $b^+$
at $x^+$ and $y^+$ and attains its unique maximum $v_1$ at $c_1$.
Similarly, on the left real arc it takes the value $b^+$ at
$y^-$ and $x^-$ and attains its unique maximum $v_1$ at $c_3$.
Hence, the four real lifts of $\ell_1$ are $[x^+,c_1], [c_3,x^-]$ in $\overline{A_1}$ and $[c_1,y^+],[y^-,c_3]$ in $\overline{A_2}$.

The coverings $A_1\to W$ and $A_2\to W$ have degrees $4$ and $2$,
respectively. Thus there are two further lifts in $\overline{A_1}$
and none in $\overline{A_2}$.
Their endpoints over $b^+$ are $\zeta$ and $\sigma(\zeta)$.
At each of $c_1,c_3$, the two real lifts already account for
the local degree $2$, so the remaining lifts must end at the
simple points $s_1,s_2$ of $g_3^{-1}(v_1)$.
After relabelling these points if necessary, the two remaining
lifts join $\zeta$ to $s_1$ and $\sigma(\zeta)$ to $s_2$,
and are exchanged by $\sigma$.

Since the pieces agree on the gluing circles, the lifts of $\ell_0$ and $\ell_1$ match in pairs over $b^+$ and assemble into exactly two connected components
\[
  \Gamma^a=[0,v_1]\cup\gamma^+\cup\gamma^-,
  \qquad
  \Gamma^b=[v_2,v_0],
\]
where $\gamma^+,\gamma^-$ are the complex conjugate arcs from $0$ through $\zeta,\sigma(\zeta)$ to $s_1,s_2$, the real arm $[0,v_1]$ of $\Gamma^a$ passes through $x^+,c_1,y^+$, and $\Gamma^b$ passes through $y^-,c_3,x^-$. Over an interior point of $\ell_0$, the component $\Gamma^a$ contains four of the six preimages and $\Gamma^b$ the remaining two. Hence, $\deg(F:\Gamma^a\to\ell)=4$ and $\deg(F:\Gamma^b\to\ell)=2$. Note that $0,c_1\in\Gamma^a$ and $c_3\in\Gamma^b$. See Figure \ref{fig:second}. 

\begin{figure}[htbp]
\centering
\begin{tikzpicture}
\node at (0,0){\includegraphics[width=0.7\textwidth]{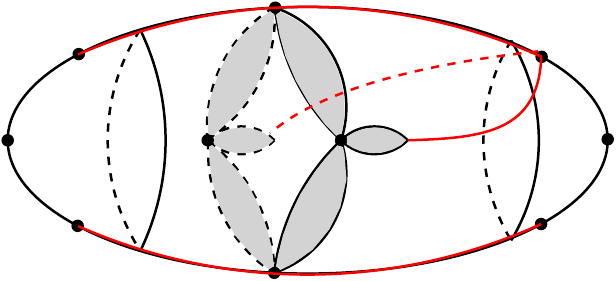}};

\node at (-3, 1.7){$\Gamma^a$};
\node at (-3, -1.8){$\Gamma^b$};
%\ruler{5}{5}
\end{tikzpicture}
\caption{$\Gamma^a$ and $\Gamma^b$.}
\label{fig:second}
\end{figure}

\emph{Step 2: the curve $\beta'$.}
Let $N$ be a small closed regular neighbourhood of $\ell$. Then $N$ is a closed disk. Since $\Gamma^b=[v_2,v_0]$ is compact and disjoint from $\ell$, we may assume, shrinking $N$ if necessary, that $N\cap P_F=\{0,v_1\}$ and $N\cap\Gamma^b=\emptyset$
(in particular $v_0,v_2\notin N$). Set
\[
  \beta':=\partial N. 
\]
It separates $\{0,v_1\}$ from $\{v_0,v_2\}$ and is essential in $\Chat\setminus P_F$.

We claim that $F^{-1}(N)$ has exactly two components $N_0,N_1$, containing $\Gamma^a$ and $\Gamma^b$ respectively, and that both are disks, mapped onto $N$ with degrees $4$ and $2$. Indeed, $\Gamma^a\cup\Gamma^b=F^{-1}(\ell)\subset F^{-1}(N)$ since $\ell\subset\operatorname{int}N$. If the components $N_0\supset\Gamma^a$ and $N_1\supset\Gamma^b$ coincided, $N_0=N_1=:U$, then $U$ would contain all six preimages of an interior point of $\ell$, hence $\deg(F|_U)=6$. But $U$ contains the critical points $0,c_1,c_3$ of local degrees $3,2,2$, so Riemann--Hurwitz would give $\chi(U)=6-\bigl((3-1)+(2-1)+(2-1)\bigr)=2$, forcing $U=\Chat$, which is absurd since $F^{-1}(N)\neq\Chat$. Hence $N_0\neq N_1$, and counting preimages of an interior point of $\ell$ gives $\deg(F:N_0\to N)\ge4$ and $\deg(F:N_1\to N)\ge2$. The degrees sum to $\deg F=6$, so equality holds and there are no further components. The only critical points of $F$ in $N_0$ are $0$ and $c_1$: the points $c_0,\infty,c_2,c_4$ map to $v_0$ or $v_2$, which lie outside $N$, and $c_3\in\Gamma^b\subset N_1$. Hence, $\chi(N_0)=4-\bigl((3-1)+(2-1)\bigr)=1$ and $\chi(N_1)=2-(2-1)=1$, so both components are disks.

Set $\beta'_1:=\partial N_0$ and $\beta'_2:=\partial N_1$. Since $\partial N$ contains no critical value of $F$,
\[
  F^{-1}(\beta')=F^{-1}(\partial N)=\partial N_0\cup\partial N_1
  =\beta'_1\cup\beta'_2,
\]
with $\deg(F,\beta'_1)=4$ and $\deg(F,\beta'_2)=2$. The disk $N_0$ contains $0,v_1$ and is disjoint from $N_1\ni v_0,v_2$, so both $\beta'_1$ and $\beta'_2$ separate $\{0,v_1\}$ from $\{v_0,v_2\}$, as does $\beta'$. For $N$ sufficiently small, $N_0$ is a regular neighbourhood of $\Gamma^a$. Deleting the two terminal branches $\gamma^+,\gamma^-$, whose interiors and terminal endpoints avoid $P_F$, does not change the homotopy class of its boundary rel $P_F$. Hence, $\beta'_1$ is homotopic to $\beta'$. Moreover, $\beta'_1$ and $\beta'_2$ are disjoint essential curves on the four-punctured sphere, so they are homotopic to each other rel $P_F$. In summary, $\beta'_1$ and $\beta'_2$ are both homotopic to $\beta'$ rel $P_F$, and $(F,\beta')$ is a folding map with $m(F,\beta')=2$. Finally, $\beta'$ induces the partition $\{0,v_1\}\,|\,\{v_0,v_2\}$, whereas $\beta$ induces $\{0,v_0\}\,|\,\{v_1,v_2\}$. The partitions differ, so $\beta'$ is not homotopic to $\beta$ rel $P_F$.
\end{proof}

\subsection{Absence of Thurston obstructions}\label{sec:obstruction}

To realize $F$ by a rational map via Theorem~\ref{thm:thurston} it remains to prove the following lemma. 

\begin{lemma}\label{lem:no-obstruction}
The Thurston map $F$ has no Thurston obstructions.
\end{lemma}

\begin{proof}
Since $\#P_F=4$, every multicurve in $\Chat\setminus P_F$ consists of a single Jordan curve (\S\ref{sec:thurston}). So a Thurston obstruction is an essential Jordan curve $\gamma$ in $\Chat\setminus P_F$ with $\lambda(\gamma)\ge1$, where $\lambda(\gamma)$ is given by \eqref{eq:lambda}. The curve $\gamma$ separates the four points of $P_F$ into two pairs, hence induces one of the three partitions
\[
  \{0,v_0\}\,|\,\{v_1,v_2\},\qquad
  \{0,v_1\}\,|\,\{v_0,v_2\},\qquad
  \{0,v_2\}\,|\,\{v_0,v_1\}.
\]
We prove that $\lambda(\gamma)<1$ in each case. In each case we denote by $V_0$ the component of $\Chat\setminus\gamma$ containing $0$ and by $V_1$ the other component, and each component of $\Chat\setminus F^{-1}(\gamma)$ is called \emph{$V_0$-type} or \emph{$V_1$-type} according to its image. More precisely, each $V_0$-type component (resp.\ $V_1$-type component) $U$ is a component of $F^{-1}(V_0)$ (resp.\ $F^{-1}(V_1)$), and the restriction $F:U\to V_0$ (resp.\ $V_1$) is a proper branched covering.

\vskip 0.3cm
\noindent\emph{Case 1: $\gamma$ separates $\{0,v_2\}$ from $\{v_0,v_1\}$.} Suppose that some component $\gamma'$ of $F^{-1}(\gamma)$ is essential and homotopic to $\gamma$ rel $P_F$, and let $\Omega$ be the component of $\Chat\setminus\gamma'$ containing $\{0,v_2\}$. Choose an arc $\tau$ joining $0$ to $v_2$ whose interior lies in $V_0\setminus P_F$. Then $\tau\cap\gamma=\emptyset$, and $\operatorname{int}\tau$ contains no critical value (Lemma~\ref{lem:F-properties}(2)). Hence $F^{-1}(\tau)$ consists of six arcs, each joining a point of $F^{-1}(0)$ to a point of $F^{-1}(v_2)$. Every such arc is disjoint from $\gamma'$: its interior maps to $\operatorname{int}\tau$, which does not meet $\gamma$, and its endpoints map into $P_F$, whereas $\gamma'$ maps to $\gamma$, which avoids $P_F$. Thus both endpoints of each lift lie in the same component of $\Chat\setminus \gamma'$, so the number of arc-ends over $0$ lying in $\Omega$ equals the number of arc-ends over $v_2$ lying in $\Omega$. On the one hand, $F^{-1}(0)\cap \Omega=\{0,v_2\}$: the local degrees at these points are $3$ and $1$ (Remark~\ref{rem:fibres}), so the number of ends over $0$ in $\Omega$ is $3+1=4$. On the other hand, the ends over $v_2$ are located at $c_2$ and $c_4$, both of local degree $3$, so their number in $\Omega$ is $3\cdot \#(\Omega\cap\{c_2,c_4\})\in\{0,3,6\}$. Since $4\notin\{0,3,6\}$, this is a contradiction. Hence, no component of $F^{-1}(\gamma)$ is essential and homotopic to $\gamma$ rel $P_F$, and $\lambda(\gamma)=0$.

\vskip 0.3cm
\noindent\emph{Case 2: $\gamma$ separates $\{0,v_0\}$ from $\{v_1,v_2\}$.} If $\gamma$ is homotopic to $\beta$ rel $P_F$, then the components of $F^{-1}(\gamma)$ are, up to homotopy rel $P_F$, those of $F^{-1}(\beta)$ with the same degrees, so $$\lambda(\gamma)=\lambda(\beta)=\frac14+\frac12=\frac34<1$$ by Lemma~\ref{lem:F-properties}(4). Assume from now on that $\gamma$ is not homotopic to $\beta$ rel $P_F$. Homotoping $\gamma$ relative to $P_F$, which does not affect $\lambda(\gamma)$, we may assume that $\gamma$ and $\beta$ are in minimal position. Set $$n:=\#(\gamma\cap\beta)=\textup{i}(\gamma,\beta).$$ Then $n\ge1$ since disjoint essential curves on the four-punctured sphere are homotopic.

Let $\gamma'_1,\dots,\gamma'_k$ be the components of $F^{-1}(\gamma)$ that are homotopic to $\gamma$ rel $P_F$, and write $d_i:=\deg(F:\gamma'_i\to\gamma)$. Then $$\lambda(\gamma)=\sum_{i=1}^{k}\frac{1}{d_i},$$ and if $k=0$ there is nothing to prove. Every $x\in\gamma\cap\beta$ is a regular value of $F$, since the critical values form the set $P_F$ and $\gamma\cap P_F=\emptyset$. Counting the preimages of $x$ on the two components of $F^{-1}(\beta)=\beta_1\cup\beta_2$, whose degrees are $4$ and $2$ (Lemma~\ref{lem:F-properties}(4)), gives
\begin{equation}\label{eq:count-total}
  \#\bigl(F^{-1}(\gamma)\cap\beta_1\bigr)=4n,
  \qquad
  \#\bigl(F^{-1}(\gamma)\cap\beta_2\bigr)=2n .
\end{equation}
On the other hand, $x$ has exactly $d_i$ preimages on $\gamma'_i$, all contained in $F^{-1}(\beta)=\beta_1\cup\beta_2$. Hence,
\begin{equation}\label{eq:count-degree}
  \#(\gamma'_i\cap\beta_1)+\#(\gamma'_i\cap\beta_2)=n\,d_i,
  \qquad i=1,\dots,k .
\end{equation}
Since $\gamma'_i$ is homotopic to $\gamma$ and $\beta_j$ is homotopic to $\beta$ rel $P_F$,
\begin{equation}\label{eq:count-lower}
  \#(\gamma'_i\cap\beta_j)\ \ge\ \textup{i}(\gamma'_i,\beta_j)=\textup{i}(\gamma,\beta)=n,
  \qquad j=1,2.
\end{equation}
Combining \eqref{eq:count-degree} and \eqref{eq:count-lower} gives $n\,d_i\ge2n$, that is $d_i\ge2$. Combining \eqref{eq:count-lower} for $j=2$ with \eqref{eq:count-total} gives $k\,n\le2n$, that is $k\le2$. Consequently
\begin{equation}\label{eq:lambda-bound}
  \lambda(\gamma)=\sum_{i=1}^{k}\frac1{d_i}\ \le\ 1,
\end{equation}
and equality holds only if $k=2$ and $d_1=d_2=2$. We exclude this equality case using two elementary observations.

(a) \emph{Alternation.} Recall that a component of $\Chat\setminus F^{-1}(\gamma)$ is called $V_0$-type (resp. $V_1$-type) if it maps to $V_0$ (resp. $V_1$) by $F$. No critical point of $F$ lies on $F^{-1}(\gamma)$, since $\gamma\cap P_F=\emptyset$. Hence, $F$ is a local homeomorphism along each component $\gamma'$ of $F^{-1}(\gamma)$, mapping a neighbourhood of $\gamma'$ onto a neighbourhood of $\gamma$. It sends the two sides of $\gamma'$ to opposite sides of $\gamma$, so the two components of $\Chat\setminus F^{-1}(\gamma)$ adjacent to $\gamma'$ are of opposite types. Moreover, the degree of a proper branched covering between oriented surfaces with boundary equals the sum of the covering degrees of its boundary circles. Applied to a $V_1$-type component $U$ this gives
\begin{equation}\label{eq:boundary-degree}
  \deg(F:U\to V_1)=\sum_{\gamma'\subset\partial U}
  \deg(F:\gamma'\to\gamma) .
\end{equation}

(b) \emph{Postcritical content.} The fibre over $0$ is $F^{-1}(0)=\{0,v_0,v_1,v_2\}=P_F$ and $F(P_F)=\{0\}\subset V_0$ (Remark~\ref{rem:fibres}, Lemma~\ref{lem:F-properties}(2)). Hence, every $V_0$-type component contains a point of $P_F$: it maps onto $V_0\ni0$, so it contains a point of $F^{-1}(0)=P_F$. No $V_1$-type component contains a point of $P_F$: such a point would map to $0\in V_0$, while its component maps onto $V_1$.

\vskip 0.3cm
Suppose now that $k=2$ and $d_1=d_2=2$. The curves $\gamma'_1$ and $\gamma'_2$ are disjoint, essential, and homotopic to each other rel $P_F$. Hence, they cobound an annulus $H$ with $H\cap P_F=\emptyset$. By observation (b), every component of $\Chat\setminus F^{-1}(\gamma)$ meeting $H$ is a $V_1$-type component. If $H$ contained a component of $F^{-1}(\gamma)$, that curve would have a $V_1$-type component on each of its two sides, contradicting observation (a). Hence, $H$ itself is a single $V_1$-type component $U$, and \eqref{eq:boundary-degree} gives
\begin{equation}\label{eq:degM}
  \deg(F:U\to V_1)=d_1+d_2=4 .
\end{equation}
But $v_2\in V_1$, and the fibre over $v_2$ is $F^{-1}(v_2)=\{c_2,c_4\}$, both points having local degree $3$ (Remark~\ref{rem:fibres}). Since $F:U\to V_1$ is a proper branched covering of degree $\deg(F:U\to V_1)$, the point $v_2$ has exactly that many preimages in $U$ counted with multiplicity. As the preimages of $v_2$ are $c_2$ and $c_4$, each contributing $3$, $\deg(F:U\to V_1)=3\cdot \#(U\cap\{c_2,c_4\})\in\{0,3,6\}$, contradicting \eqref{eq:degM}. Hence, $\lambda(\gamma)<1$.

\vskip 0.3cm
\noindent\emph{Case 3: $\gamma$ separates $\{0,v_1\}$ from
$\{v_0,v_2\}$.}
We argue as in Case~2, with $\beta'$ in place of $\beta$. Indeed, if $\gamma$ is homotopic to $\beta'$ rel $P_F$, then $\lambda(\gamma)=\lambda(\beta')=\tfrac14+\tfrac12<1$ by Lemma~\ref{lem:beta-prime}. Otherwise $n':=\textup{i}(\gamma,\beta')\ge1$, and the counts \eqref{eq:count-total}--\eqref{eq:count-lower}, with the components $\beta'_1,\beta'_2$ of degrees $4$ and $2$ from Lemma~\ref{lem:beta-prime} in place of $\beta_1,\beta_2$, again yield \eqref{eq:count-degree}--\eqref{eq:lambda-bound}. Observations (a) and (b) hold as stated: their proofs use only that $\gamma\cap P_F=\emptyset$, that $F^{-1}(0)=P_F$, and that $0\in V_0$. Since $v_2\in V_1$ also in this case, the fibre count over $v_2$ excludes the equality case exactly as above.

In all three cases $\lambda(\gamma)<1$. Hence, no essential curve is a Thurston obstruction, and the proof is complete.
\end{proof}

\subsection{Realization by a rational map}\label{sec:realization}

We now realize $F$ by a rational map and complete the proof of Theorem~\ref{thm:A}.

\begin{proof}[Proof of Theorem~\ref{thm:A}]
By Lemma~\ref{lem:F-properties}(2),(5), $F$ is a Thurston map with hyperbolic orbifold, and by Lemma~\ref{lem:no-obstruction} it has no Thurston obstruction. Theorem~\ref{thm:thurston} therefore provides a rational map $f$, unique up to M\"obius conjugacy, which is Thurston equivalent to $F$: there are homeomorphisms $h_0,h_1$ of $\Chat$, isotopic rel $P_F$, such that $$h_0\circ F=f\circ h_1$$ on $\Chat$. Transporting the curves of Lemma~\ref{lem:F-properties}(4) and Lemma~\ref{lem:beta-prime} by $h_0$ gives Jordan curves, which we again denote by $\beta$ and $\beta'$, in $\Chat\setminus P_f$. Note that Thurston equivalence preserves the postcritical set, the local degrees, the separation patterns and the homotopy classes of curves relative to the postcritical set. Hence, each of $f^{-1}(\beta)$ and $f^{-1}(\beta')$ consists of two Jordan curves, essential and homotopic to $\beta$ (resp.\ $\beta'$) rel $P_f$, mapping with degrees $4$ and $2$, and $\beta,\beta'$ induce the two different partitions $\{0,v_0\}\,|\,\{v_1,v_2\}$ and $\{0,v_1\}\,|\,\{v_0,v_2\}$ of $P_f=\{0,v_0,v_1,v_2\}$. The orbifold of $f$ has the same signature $(2,2,3,\infty)$ as that of $F$. In particular it is hyperbolic, and hence is not a Latt\`es map.

Since $F$ commutes with the reflection $\sigma$ (Lemma~\ref{lem:F-properties}(3)), the conjugate $\sigma\circ f\circ\sigma$ is Thurston equivalent to $F$ as well. The uniqueness in Theorem~\ref{thm:thurston} then implies that $f$ can be chosen with real coefficients. We normalize $f$ so that the superattracting fixed point is $0$, and $\infty$ is a simple critical point.
\end{proof}

\subsection{A numerical illustration}\label{sec:formula}
Under the above normalization, we use the real rational model $f$ as a numerical illustration. 
\begin{equation}\label{eq:explicit-f}
  f(w)=\frac{P(w)}{a\,Q(w)^{3}-c\,P(w)},
\end{equation}
where
\[
  P(w)=w^{3}(cw+1)\bigl(\bigl(c-\tfrac12\bigr)w+1\bigr)
       \bigl(\bigl(c+\tfrac12\bigr)w+1\bigr),
  \quad Q(w)=(c^{2}+u)w^{2}+2cw+1. 
\]
The parameters were obtained numerically and have approximate values 
\begin{equation}\label{eq:constants}
  c\approx0.22428,\qquad
  u\approx0.55789,\qquad
  a\approx0.09955 .
\end{equation}
Figure \ref{fig:three-images} shows the numerically computed Julia set of $f$, together with the approximate critical orbits and the curves $\beta, \beta'$. 

\vskip 0.5cm
\paragraph*{Acknowledgments} The authors would like to thank Guizhen Cui and Yan Gao for helpful discussions and suggestions. The authors are partially supported by NSFC (Grant no.~12271115, 12526609, 12601135).

\end{document}